\documentclass[12pt]{article} 

\usepackage{amsmath}
\usepackage{amsthm}
\usepackage{amsfonts}
\usepackage{mathrsfs}
\usepackage{stmaryrd}
\usepackage{setspace}
\usepackage{fullpage}
\usepackage{amssymb}
\usepackage{breqn}
\usepackage{enumitem}
\usepackage{bbold} 
\usepackage{authblk}
\usepackage{comment}
\usepackage{hyperref}
\usepackage{pgf,tikz}
\usepackage{graphicx}
\usepackage{subcaption}

\newtheorem{thm}{Theorem}[section]

\newtheorem{proposition}[thm]{Proposition}

\newcommand\ex{\ensuremath{\mathrm{ex}}}

\newcommand\cH{{\mathcal H}}

\newcommand\cT{{\mathcal T}}

\newcommand{\ignore}[1]{}

\title{A note on the number of triangles in graphs without the suspension of a path on four vertices}
\author{Dániel Gerbner}

\date{\small Alfr\'ed R\'enyi Institute of Mathematics}

\begin{document}

\maketitle

\begin{abstract}
    The suspension of the path $P_4$ consists of a $P_4$ and an additional vertex connected to each of the four vertices, and is denoted by $\hat{P_4}$. The largest number of triangles in a $\hat{P_4}$-free $n$-vertex graph is denoted by $\ex(n,K_3,\hat{P_4})$. Mubayi and Mukherjee in 2020 showed that $ \ex(n,K_3,\hat{P_4})= n^2/8+O(n)$. We show that for sufficiently large $n$, $\ex(n,K_3,\hat{P_4})=\lfloor n^2/8\rfloor$.
    
\end{abstract}

\section{Introduction}

One of the main topics in extremal graph theory was initiated by Tur\'an \cite{T} and deals with the quantity $\ex(n,F)$, which is the largest number of edges in $n$-vertex  $F$-free graphs.

Generalized Tur\'an problems deal with the quantity $\ex(n,H,F)$, which is the largest number of copies of a subgraph $H$ in $n$-vertex $F$-free graphs. The systematic study of generalized Tur\'an problems was initiated by Alon and Shikhelman \cite{ALS2016}.

The \textit{suspension} $\hat{F}$ of a graph $F$ is obtained by adding an additional vertex $v$ to $F$ and for each vertex $u\in V(F)$, adding the edge $uv$.
Mubayi and Mukherjee \cite{mm} studied $\ex(n,K_3,\hat{F})$ for various different bipartite graphs $F$. In particular, they investigated $\ex(n,K_3,\hat{P_4})$, where $P_4$ denotes the path on four vertices. 

\begin{proposition}[Mubayi, Mukherjee \cite{mm}]\label{mumu} We have
$\lfloor n^2/8\rfloor-O(1) \le \ex(n,K_3,\hat{P_4})\le n^2/8+3n$.
\end{proposition}

We determine $\ex(n,K_3,\hat{P_4})$ exactly if $n$ is sufficiently large.

\begin{thm}\label{main} If $n$ is sufficently large, then
$\ex(n,K_3,\hat{P_4})=\lfloor n^2/8\rfloor$.
\end{thm}

We present the proof in Section 2. The proof uses a (special case of) a theorem of Gy\H ori \cite{gyori} on Berge triangles and progressive induction. However, we can incorporate the proof of Gy\H ori's result into our proof and use a simple version of progressive induction, this way making our proof self-contained. We introduce Berge hypergraphs, state Gy\H ori's result, show the connection to our problem and describe progressive induction in Section 3.

\section{Proof of Theorem \ref{main}}

Let us describe the lower bound first. We take an almost balanced complete bipartite graph on $n$ vertices and add a full matching to one of the parts. Then the neighborhood of every vertex is either a star or or a matching, thus this graph if $\hat{P_4}$-free. If $n=4k$, we take $K_{2k,2k}$ and add $k$ independent edges to one of the parts. If $n=4k+1$, we take $K_{2k,2k+1}$ and add $k$ independent edges to the smaller part. If $n=4k+2$, we take $K_{2k,2k+2}$ and add $k$ independent edges to the smaller part. If $n=4k+3$, we take $K_{2k+1,2k+2}$ and add $k+1$ independent edges to the larger part.

\smallskip

Let us continue with the proof of the upper bound.
 We will show that if $n\ge 525$ and $f(n):=\ex(n,K_3,\hat{P_4})-\lfloor n^2/8\rfloor>0$, then either $f(n)<f(n-1)$ or $f(n)<f(n-4)$. Let $K=\max\{f(n):n<525\}\le \binom{525}{3}$, then for $n\ge 525$ we have $f(n)<K$ (note that by Proposition \ref{mumu} we have $K\le 1575$). Similarly for $n\ge 529$ we have $f(n)<K-1$ and in general for $n\ge 525+4k$ we have $f(n)<K-k$. Therefore, for $n\ge 525+4K$ the statement holds.
 
Let $G$ be a $\hat{P_4}$-free $n$-vertex graph with $\ex(n,K_3,\hat{P_4})$ triangles.
Assume first that $G$ is $K_4$-free. Observe that every triangle in $G$ contains at least two edges that are not contained in any other triangle. Indeed, assume that $a,b,c$ forms a triangle and $ab$ and $bc$ are contained in other triangles $a,b,x$ and $b,c,y$. If $x=y$, then there is a $K_4$ with vertices $a,b,c,x$, and if $x\neq y$, then each vertex of the path $xacy$ is connected to $b$. 

Let us pick for each triangle in $G$ two edges that are not contained in any other triangle (if there are three such edges, we pick two arbitrarily). The resulting graph $G'$ clearly has twice as many edges as the number of triangles in $G$. Observe that $G'$ is triangle-free, since $G'$ is a subgraph of $G$, and for every triangle of $G$, one of its edges was not picked. Therefore, by Mantel's theorem $G'$ has at most $\lfloor n^2/4\rfloor$ edges, hence $G$ has at most $\lfloor n^2/8\rfloor$ triangles.

Assume now that $G$ contains a $K_4$ with vertices $U=\{v_1,v_2,v_3,v_4\}$.  
Assume first that there is a vertex $v$ that is contained in at most $(2n-5)/8$ triangles.
 Then  $\ex(n,K_3,\hat{P_4})\le\ex(n-1,K_3,\hat{P_4})+(2n-5)/8$, thus $f(n)-f(n-1)=\ex(n,K_3,\hat{P_4})-\lfloor n^2/8\rfloor- \ex(n-1,K_3,\hat{P_4})+\lfloor (n-1)^2/8\rfloor<0$. As there is no $P_4$ in the neighborhood $N(v)$ of $v$ in $G$, the connected components of $G[N(v)]$ are stars or triangles. This means that there are at most $|N(v)|$ edges inside $N(v)$, thus $v$ is in at most $|N(v)|$ triangles. Therefore, we are done if there is a vertex
with degree at most $(2n-5)/8$.
Assume now that there are at most $n-3$ triangles containing a vertex from $U$. Then $\ex(n,K_3,\hat{P_4})\le\ex(n-4,K_3,\hat{P_4})+n-3$, thus $f(n)<f(n-1)$.

Assume from now on that every vertex has degree more than $(2n-5)/8$ (thus at least $(n-4)/4$), and for every $K_4$, there are more than $n-3$ triangles that contain at least one of its vertices.
Observe that every vertex outside $U$ is connected to at most one vertex of $U$ by the $\hat{P_4}$-free property. Let $V_i=N(v_i)\setminus U$ for $i\le 4$ and $V_5=V(G)\setminus (U\cup V_1\cup V_2\cup V_3\cup V_4)$, then the $V_i$'s are pairwise disjoint, thus there are at most $n-4$ triangles containing one vertex from $U$. There are 4 triangles inside $U$ and no triangles containing exactly two vertices of $U$, thus there are $n$ or $n-1$ or $n-2$ triangles containing at least one vertex from $U$. In particular, at least two of the $v_i$'s have the property that $v_i$ is in $|N(v_i)|$ triangles. This can happen only if $G[V_i]$ consists of vertex-disjoint triangles. Without loss of generality this holds for $V_1$ and $V_2$. 

Observe that the above holds for every $K_4$, i.e. there are two vertices in every $K_4$ such that their neighborhood induces vertex disjoint triangles.  In particular, there is a such vertex $v\in V_1$. Observe that only one of the triangles in $N(v)$ intersects $V_1$, and every triangle from $N(v)$ contains at most one vertex from $V_2$. Indeed, if $a,b,c$ form such a triangle, with $a,b\in V_2$, then there is a triangle $a,b,c'$ inside $V_2$. Then $vcac'$ is a $P_4$ in the neighborhood in $b$,
a contradiction. 

Recall that $v$ has degree at least $(n-4)/4$ and its neighborhood consists of at least $(n-4)/12$ vertex disjoint triangles.
This means that out of the 
at least $(n-4)/12$ triangles in $N(v)$, one is inside $V_1\cup \{v_1\}$, and the other at least $(n-16)/12$ triangles each contain an edge from $V_3\cup V_4\cup V_5$. Observe that these are independent edges.

Let $A$ be denote the set of vertices in $G$ that are contained in some $K_4$. Then $|A|\ge |U|+|V_1|+|V_2|+(n-16)/6\ge (8n+8)/12$. Since for $i\le 4$ we have $|V_i|\ge (n-16)/4$ but the union of these disjoint four sets have order at most $n-4$, we have that $|V_i|\le (n+32)/4$. Therefore, the degree of vertices in $A$ is at most $(n+44)/4$. Let $B$ be denote the set of vertices in $G$ that are not contained in any $K_4$. A vertex of $B$ is connected to at most one vertex in each triangle inside $V_1$ and $V_2$, and to at most one vertex of each of the $(n-16)/12$ independent edges inside $V_3\cup V_4\cup V_5$ that form triangles with $v$. This means that the degree of a vertex of $B$ is at most $n-1-(n-3)/3-(n-6)/12=(7n+6)/12$. Therefore, the sum of the degrees in $G$ is at most $\frac{8n+8}{12}\frac{n+44}{4}+\frac{4n-8}{12}\frac{7n+6}{12}=\frac{13n^2+262n+252}{36}$. Using that the number of triangles containing a vertex is at most its degree, it gives an upper bound on three times the number of triangles in $G$ (as we count every triangle three times this way). Therefore, the number of triangles is at most $\frac{13n^2+262n+252}{108}<\lfloor n^2/8\rfloor$ if $n\ge 525$, completing the proof.

\section{Remarks}
 Given a graph $F$, a hypergraph $\cH$ is a \textit{Berge copy of $F$} (in short Berge-$F$) if $V(F)\subset V(\cH)$ and there is a bijection $\phi:E(F)\rightarrow E(\cH)$ such that for each edge $e\in E(F)$ we have $e\subset \phi(e)$. In other words, we can obtain $\cH$ from $F$ by enlarging the edges arbitrarily. Note that there may be several non-isomorphic Berge copies of $F$. We denote by $\ex_r(n,\text{Berge-}F)$ the largest number of hyperedges in an $r$-uniform $n$-vertex hypergraph that does not contain any of the Berge copies of $F$.

Berge hypergraphs were introduced by Gerbner and Palmer \cite{gp1}, extending the well-studied notion of Berge cycles and paths. The connection of Berge hypergraphs to generalized Tur\'an problems was shown in \cite{gp2}, by the following simple proposition: $\ex(n,K_r,F)\le \ex_r(n,\text{Berge-}F)\le \ex(n,K_r,F)+\ex(n,F)$. We remark that the lower bound follows by observing that the hypergraph consisting of the vertex sets of copies of $K_r$ in an $F$-free graph is Berge-$F$-free. We will connect our problem to Berge hypergraphs differently.

Given a graph $G$, let us denote by $\cT(G)$ the 3-uniform hypergraph with vertex set $V(G)$, where $\{a,b,c\}$ is a hyperedge if and only if there is a triangle with vertices $a,b,c$ in $G$.

\begin{proposition}
$\cT(G)$ is Berge-$K_3$-free if and only if $G$ is $\hat{P_4}$-free and $K_4$-free.
\end{proposition}

\begin{proof}
Let us assume that there is a Berge-$K_3$ in $\cT(G)$. It consists of 3 hyperedges that correspond to a graph triangle with vertices $a,b,c$. Then the 3 hyperedges are $\{a,b,x\}$, $\{b,c,y\}$ and $\{c,a,z\}$. At most one of these hyperedges are $\{a,b,c\}$, thus without loss of generality $x,y\not\in \{a,b,c\}$. The edges $ab,ax,bx,bc,by,cy,ca,cz,az$ are all in $G$. If $x=y$, then there is a $K_4$ in $G$ with vertices $a,b,c,x$. If $x\neq y$, then the vertices of the path $xacy$ are each connected to $b$, giving a $\hat{P_4}$.

If $G$ contains a $K_4$ with vertices $a,b,c,d$, then there is a bijection from the edges of the triangle $a,b,c$ to the hyperedges $\{a,b,d\}$, $\{b,c,d\}$, $\{c,a,d\}$ of $\cT(G)$, giving us a Berge-$K_3$ in $\cT(G)$. If $G$ contains a path $xacy$ such that is vertices are each connected to a fifth vertex $b$, then there is a bijection from the edges of the triangle $a,b,c$ to the hyperedges $\{a,b,x\}$, $\{b,c,y\}$, $\{c,a,b\}$ of $\cT(G)$, giving us a Berge-$K_3$ in $\cT(G)$.
\end{proof}

The above proposition immediately implies that $\ex(n,K_3,\{\hat{P_4},K_4\})\le \ex_3(n,\text{Berge-}K_3)$.
Gy\H ori \cite{gyori} showed that  $\ex_3(n,\text{Berge-}K_3)=\lfloor n^2/8\rfloor$. We remark that the lower bound is obtained by $\cT(G)$, where $G$ is the construction that gives the lower bound in Theorem \ref{main}. 
Using the above results, we immediately obtain that $\ex(n,K_3,\{\hat{P_4},K_4\})= \ex_3(n,\text{Berge-}K_3)=\lfloor n^2/8\rfloor$. To strengthen it and obtain Theorem \ref{main}, we need to show that if a $\hat{P_4}$-free graph contains a $K_4$, then it contains at most $\lfloor n^2/8\rfloor$ triangles.

We remark that the proof of the upper bound of Proposition \ref{mumu} in \cite{mm} also deals separately with the two cases depending on whether a $\hat{P_4}$-free graph is $K_4$-free or not. Our improvement in the first case is by the above observations, and Gy\H ori's proof is incorporated to our proof in Section 2. The second case is dealt by induction in \cite{mm}. However, we cannot prove our improved bound in a similar way, as our Theorem \ref{main} does not hold for every $n$. For example, if $n=7$, then $\lfloor n^2/8\rfloor=6$, but two copies of $K_4$ sharing a vertex is a $\hat{P_4}$-free 7-vertex graph and contains 8 triangles. Therefore, we do not have the base step of the induction, even though the induction step works almost the same as in \cite{mm}. 

Progressive induction, introduced by Simonovits \cite{sim}, is used in such cases (it was first used in the generalized Tur\'an setting in \cite{ger}). The basic idea is to have a stronger induction step, where we also deal with the case of equality. This means that for small values of $n$ $\ex(n,K_3,\hat{P_4})$ can be larger than $\lfloor n^2/8\rfloor$, but this surplus starts decreasing and eventually vanishes. The first paragraph of the proof of the upper bound of Theorem \ref{main} is a simple example of this; we do not go into the details of the general version. Finally, we mention that it is likely that Theorem \ref{main} holds for every $n\ge 8$. If the upper bound for $n=8,9,10,11$ is proved, either by a careful case analysis or by brute force, then ordinary induction could be used instead of progressive induction to complete the proof of Theorem \ref{main}.


\bigskip

\textbf{Funding}: Research supported by the National Research, Development and Innovation Office - NKFIH under the grants KH 130371, SNN 129364, FK 132060, and KKP-133819.

\end{document}